\documentclass[12pt]{article}
\usepackage{makeidx}
\usepackage{amssymb}
\usepackage{amsmath}
\usepackage{amsfonts}
\usepackage{mitpress}

\newtheorem{theorem}{Theorem}

\newtheorem{claim}[theorem]{Lemma and Definition}
\newtheorem{conclusion}[theorem]{Corollary and Definition}

\newtheorem{definition}[theorem]{Definition}
\newtheorem{example}[theorem]{Example}

\newtheorem{lemma}[theorem]{Lemma}
\newtheorem{notation}[theorem]{Notation}

\newtheorem{proposition}[theorem]{Proposition}
\newtheorem{remark}[theorem]{Remark}

\newenvironment{proof}[1][Proof]{\noindent\textbf{#1.} }{\ \rule{0.5em}{0.5em}}
\newdimen\dummy
\dummy=\oddsidemargin
\input{tcilatex}
\begin{document}

\title{On the Closed Graph Theorem and the Open Mapping Theorem}
\author{Henri Bourl\`{e}s\thanks{%
Satie, ENS de Cachan/CNAM , 61 Avenue Pr\'{e}sident Wilson, F-94230 Cachan,
France. e-mail: henri.bourles@satie.ens-cachan.fr}}
\maketitle

\begin{abstract}
Let $E,F$ be two topological spaces and $u:E\rightarrow F$ be a map. \ If $F$
is Haudorff and $u$ is continuous, then its graph is closed. \ \ The Closed
Graph Theorem establishes the converse when $E$ and $F$ are suitable objects
of topological algebra, and more specifically topological groups,
topological vector spaces (TVS's) or locally vector spaces (LCS's) of a
special type. The Open Mapping Theorem, also called the Banach-Schauder
theorem, states that under suitable conditions on $E$ and $F,$ if $%
v:F\rightarrow E$ is a continuous linear surjective map, it is open. \ When
the Open Mapping Theorem holds true for $v,$ so does the Closed Graph
Theorem for $u.$ \ The converse is also valid in most cases, but there are
exceptions. \ This point is clarified. Some of the most important versions
of the Closed Graph Theorem and of the Open Mapping Theorem are stated
without proof but with the detailed reference.
\end{abstract}

\sloppy

\section{Introduction}

Let $E,F$ be two topological spaces and $u:E\rightarrow F$ be a map. \ If $F$
is Haudorff and $u$ is continuous, then its graph is closed (see Lemma \ref%
{lemma-continuous-then-closed} below). \ The Closed Graph Theorem
establishes the converse when $E$ and $F$ are suitable objects of
topological algebra, and more specifically topological groups, topological
vector spaces (TVS's) or locally vector spaces (LCS's) of a special type. \
Banach's theorem states that when $E$ and $F$ are Fr\'{e}chet spaces and $u$
is linear, this map is continuous if, and only if its graph is closed (\cite%
{Banach Operations lineaires}, p. 41, Thm. 7). \ The Open Mapping Theorem,
also called the Banach-Schauder theorem, states that under suitable
conditions on $E$ and $F,$ if $v:F\rightarrow E$ is a continuous linear
surjective map, it is open, i.e. for every open subset $\Omega $ of $F,$ $%
v\left( \Omega \right) $ is an open subset of $E$ (\cite{Schauder}, p. 6,
Satz 2). \ These two results are easy consequences of Banach's isomorphism
theorem (\cite{Banach-isomorphism}, p. 238, Thm. 7). \ All these results
still hold true when $E,F$ are metrizable complete TVS's over a non-discrete
valued division ring $\mathbf{k}$ \cite{Bourbaki-EVT}.

Several generalizations of the Closed Graph Theorem and the Open Mapping
Theorem have been established when $\mathbf{k}=%
%TCIMACRO{\U{211d} }%
%BeginExpansion
\mathbb{R}
%EndExpansion
$ or $%
%TCIMACRO{\U{2102} }%
%BeginExpansion
\mathbb{C}
%EndExpansion
$ and when $E,F$ are special kinds of LCS's over $\mathbf{k}$, or when $E,F$
are suitable topological groups. \ When the Open Mapping Theorem holds true
for $v,$ so does the Closed Graph Theorem for $u.$ \ The converse is also
valid in most cases, but there are exceptions, e.g. Pt\'{a}k's classical
result \cite{Ptak} which can be stated as follows: (i) Let $E$ be barrelled
and $F$ be infra-Pt\'{a}k (i.e. $B_{r}$-complete). \ A $\mathbf{k}$-linear
map $u:E\rightarrow F$ is continuous if its graph is closed. \ (ii) Let $F$
be Pt\'{a}k (i.e. $B$-complete) and $E$ be barrelled, then a $\mathbf{k}$%
-linear continuous surjection $v:F\rightarrow E$ is open. \ A Pt\'{a}k space
is infra-Pt\'{a}k (\cite{Ptak}, (4.2)), but Valdivia has shown in \cite%
{Valdivia-84} that there exist infra-Pt\'{a}k spaces which are not Pt\'{a}k.
\ If "Pt\'{a}k" is replaced by "infra-Pt\'{a}k" in (ii), this statement is
no longer correct. \ 

We review in this paper some of the most important versions of the Closed
Graph Theorem and of the Open Mapping Theorem. Their connection is expressed
in Theorem \ref{Main thm} below (Section \ref{sect-main-thm}). \ Some
preliminaries are given in Section \ref{Sect-preliminaires}. \ A typology of
topological spaces, topological groups and topological LCS's is given in
Section \ref{Sect-typology}. \ Versions of the Closed Graph Theorems and of
the Open Mapping Theorems in topological groups and in LCS's are stated in
Sections \ref{Sect-thm-top-groups} and \ref{sect-thm-LCS} respectively.

\section{Preliminaries\label{Sect-preliminaires}}

\subsection{Continuity and closed graph}

In this paper, we use nets \cite{Kelley} rather than filters \cite%
{Bourbaki-TG-I} in topological spaces. \ These two notions can be used in an
equivalent way for convergence studies, and they are connected as follows:\
denoting by $I$ a preordered set, with preorder relation $\succeq $, a net $%
\left( x_{i}\right) _{i\in I}$ in $E$ is a map $I\rightarrow E:i\mapsto
x_{i}.$ \ The family of sets $\mathfrak{B}=\left( B_{i}\right) _{i\in I}$
where $B_{i}=\left\{ x_{i}\in E;j\succeq i\right\} $ is a base of the
elementary filter associated with $\left( x_{i}\right) _{i\in I},$ and $%
\left( x_{i}\right) $ converges to $x\in E$ if, and only if $\lim \mathfrak{B%
}=x.$ \ Conversely, let $\mathfrak{B}$ be a filter base and, using the
choice axiom, for every $B\in \mathfrak{B}$ let $x_{B}\in B.$ \ Then $\left(
x_{B}\right) _{B\in \mathfrak{B}}$ is a net associated to $\mathfrak{B}$ and 
$\lim \mathfrak{B}=x$ if, and only if every net associated to $\mathfrak{B}$
converges to $x.$ \ A sequence is a net with a countable set of indices
which can be identified with $%
%TCIMACRO{\U{2115} }%
%BeginExpansion
\mathbb{N}
%EndExpansion
.$

Let $E,F$ be two topological spaces and $u:E\rightarrow F$ be a map. \ The
map $u$ is continuous (resp. sequentially continuous) if, and only if for
any point $x\in E,$ whenever $\left( x_{i}\right) $ is a net (resp. a
sequence) in $E$ converging to $x,$ the net (resp. the sequence) $\left(
u\left( x_{i}\right) \right) $) converges to $u\left( x\right) .$ \ The
graph $\limfunc{Gr}\left( u\right) $ of $u$ is closed (resp. sequentially
closed) if, and only if whenever $\left( x_{i},u\left( x_{i}\right) \right) $
is a net (resp. a sequence) in $E\times F$ converging to $\left( x,y\right)
, $ necessarily $\left( x,y\right) \in \limfunc{Gr}\left( u\right) ,$ i.e. $%
y=u\left( x\right) .$ \ If $E$ is a topological group, the point $x$ here
above can be replaced by the neutral element $e.$

\begin{lemma}
Let $E,F$ be two topological spaces, $u:E\rightarrow F$ be a map, and $%
\limfunc{Gr}\left( u\right) $ be its graph. \ Then $u$ is continuous if, and
only if the map $p:\limfunc{Gr}\left( u\right) \rightarrow E:\left(
x,f\left( x\right) \right) \mapsto x$ is open.
\end{lemma}

\begin{proof}
The map $p$ is bijective with inverse $p^{-1}:x\mapsto \left( x,f\left(
x\right) \right) .$ \ Thus $p^{-1}$ is continuous (i.e., $p$ is open) if,
and only if $f$ is continuous.
\end{proof}

The following result is a slight generalization of (\cite{Bourbaki-TG-I}, p.
I.53, Cor. 2 of Prop. 2):

\begin{lemma}
\label{lemma-continuous-then-closed}Let $E,F$ be two topological spaces and $%
u:E\rightarrow F$ be a map. \ If $F$ is Hausdorff and $u$ is continuous
(resp. sequentially continuous), then $\limfunc{Gr}\left( u\right) \ $is
closed (resp. sequentially closed) in $E\times F.$
\end{lemma}

\begin{proof}
Let $\left( \left( x_{i},u\left( x_{i}\right) \right) \right) _{i\in I}$ be
a net (resp. a sequence) in $\limfunc{Gr}\left( u\right) $ converging to $%
\left( x,y\right) \in E\times F.$ \ Then, $\left( x_{i}\right) \rightarrow
x. $ \ If $u$ is continuous (resp. sequentially continuous), this implies $%
\left( u\left( x_{i}\right) \right) \rightarrow u\left( x\right) .$ \ If $F$
is Hausdorff, $u\left( x\right) =y,$ thus $\left( x,y\right) \in \limfunc{Gr}%
\left( u\right) $ and $\limfunc{Gr}\left( u\right) $ is closed (resp.
sequentially closed).
\end{proof}

The Closed Graph Theorem is the converse of Lemma \ref%
{lemma-continuous-then-closed} when $E,$ $F$ and $u$\ satisfy the suitable
conditions.

Our first example will be:

\begin{lemma}
\label{lemma-closed-graph-F-compact}Let $E,F$ be two topological spaces, $%
u:E\rightarrow F$ be a map, and assume that $F$ is (Hausdorff and) compact.
\ Then $u$ is continuous if, and only if $\limfunc{Gr}\left( u\right) $ is
closed.
\end{lemma}

\begin{proof}
(If): Let $\left( x_{i}\right) $ be a net converging to $x$ in $E.$ \ Since $%
F$ is compact, $\left( u\left( x_{i}\right) \right) $ has a cluster point in 
$F,$ thus $\left( \left( x_{i},u\left( x_{i}\right) \right) \right) $ has a
cluster point $\left( x,y\right) $ in $E\times F.$ \ Since $\limfunc{Gr}%
\left( u\right) $ is closed, $\left( x,y\right) \in \limfunc{Gr}\left(
u\right) ,$ i.e. $y=u\left( x\right) .$ \ Therefore $u\left( x\right) $ is
the unique cluster point of $\left( u\left( x_{i}\right) \right) $ and $%
\left( u\left( x_{i}\right) \right) \rightarrow u\left( x\right) ,$ which
implies that $u$ is continuous.
\end{proof}

For the second example we will give, the following is necessary:

\begin{definition}
A subset $A$ of a topological space $X$ is called \emph{nearly open} if $%
A\subset \overset{o}{\bar{A}},$ i.e. $A$ belongs to the interior of its
closure. \ Let $u:X\rightarrow Y$ where $X$ and $Y$ are topological spaces.
\ The map $u$ is called \emph{nearly open} (resp. \emph{nearly continuous})
if for any open set $\Omega $ in $X$ (resp. $Y$), $u\left( \Omega \right) $
(resp. $u^{-1}\left( \Omega \right) $) is nearly open. \ (In the above, the
term "almost" is sometimes used in place of "nearly".)
\end{definition}

The following result, the proof of which is less trivial than that of Lemma %
\ref{lemma-closed-graph-F-compact}, is due to Weston and Pettis \cite{Weston}%
, \cite{Pettis}:

\begin{theorem}
Let $X,Y$ be two metrizable complete topological spaces. Let $u:X\rightarrow
Y$. \ The following conditions are equivalent:\newline
(a) $u$ is continuous.\newline
(b) $\limfunc{Gr}\left( u\right) $ is closed and $u$ is nearly continuous.
\end{theorem}

\subsection{Inductive limits and locally convex hulls}

\begin{notation}
\label{not-Tcat}$\mathbf{TCat}$ is a category used in topological algebra
and is a subcategory of $\mathbf{TSet}$, the category of topological spaces (%
\cite{Bourbaki-TG-I}, Chap. II); there exists a forgetful functor $\mathfrak{%
Forget}:\mathbf{TCat}\rightarrow \mathbf{Cat}$ where $\mathbf{Cat}$ is a
subcategory of $\mathbf{Set}$, the category of sets. \ The three cases
considered below are: \newline
(1) $\mathbf{TGrp}\overset{\mathfrak{Forget}}{\rightarrow }\mathbf{Grp}$
where $\mathbf{TGrp}$ is the category of topological groups, and $\mathbf{Grp%
}$ is the category of groups; \newline
(2) $\mathbf{TVsp}_{\mathbf{k}}\overset{\mathfrak{Forget}}{\longrightarrow }%
\mathbf{Vsp}_{\mathbf{k}}$ where $\mathbf{TVsp}_{\mathbf{k}}$ the category
of right topological spaces over a topological division ring $\mathbf{k}$
and $\mathbf{Vsp}_{\mathbf{k}}$ is the category of right vector spaces over
the division ring $\mathbf{k})$;\newline
(3) $\mathbf{LCS}\overset{\mathfrak{Forget}}{\rightarrow }\mathbf{Vsp}_{%
\mathbf{k}}$ where $\mathbf{LCS}$ the category of locally convex spaces over 
$\mathbf{k}=%
%TCIMACRO{\U{211d} }%
%BeginExpansion
\mathbb{R}
%EndExpansion
$ or $%
%TCIMACRO{\U{2102} }%
%BeginExpansion
\mathbb{C}
%EndExpansion
$;\newline
For short, a morphism of $\mathbf{Cat}$ (resp. $\mathbf{TCat}$) is called a 
\emph{map} (resp. a \emph{morphism}).
\end{notation}

Let us recall the general notions of inductive (or direct) limit and of
filtrant inductive limit (\cite{Kashiwra-Schapira}, Chap 2 \& 3).

Let $\mathcal{C}$ be a category and $\left( F_{\alpha }\right) _{\alpha \in
A}$ be a family of objects of $\mathcal{C}.$ \ An inductive limit $%
\lim\limits_{\longrightarrow }F_{\alpha }$ is (if it exists) a functor $%
\mathcal{C}\rightarrow \mathcal{C}$, uniquely determined up to isomorphism,
consisting of one object and a family of morphisms $\varphi _{\alpha
}:F_{\alpha }\rightarrow \lim\limits_{\longrightarrow }F_{\alpha }$ such
that for any object $X$ and morphisms $f_{\alpha }:F_{\alpha }\rightarrow X,$
there exists a morphism $f:\lim\limits_{\longrightarrow }F_{\alpha
}\rightarrow X$ such that $f_{\alpha }$ factors through $f$ according to $%
f_{\alpha }=f\circ \varphi _{\alpha }.$

Let $A$ be equipped with a binary relation $\preccurlyeq .$ The latter is a 
\emph{preorder relation} if (and only if) it is transitive and reflexive. \
The preordered set $A$ is called \emph{right filtrant }if each pair of
elements has an upper bound. \ Let $\left( F_{\alpha }\right) _{\alpha \in
A} $ be a family of objects of $\mathcal{C}$ indexed by the right-filtrant
set $A$ and, for each pair $\left( \alpha ,\beta \right) \in A\times A$ such
that $\alpha \preccurlyeq \beta ,$ let $\varphi _{\beta }^{\alpha
}:F_{\alpha }\rightarrow F_{\beta }$ be a morphism such that $\varphi
_{\alpha }^{\alpha }=id_{F_{\alpha }}$ and $\varphi _{\gamma }^{\alpha
}=\varphi _{\gamma }^{\beta }\circ \varphi _{\beta }^{\alpha }$ for all $%
\alpha ,\beta ,\gamma \in A$ such that $\alpha \preccurlyeq \beta
\preccurlyeq \gamma .$ \ Then $\left( F_{\alpha },\varphi _{\beta }^{\alpha
}\right) $ is called a \emph{direct system} with index set $A.$ \ The
inductive limit $G=\lim\limits_{\longrightarrow }G_{\alpha }$ (if it exists)
is called \emph{filtrant} if the compatibility relation $\varphi _{\alpha
}=\varphi _{\beta }\circ \varphi _{\beta }^{\alpha }$ holds whenever $\alpha
\preceq \beta .$

\subsubsection{Inductive limits of topological groups}

Let $\mathbf{TCat}=\mathbf{TGr}$. \ A topology on a group $G$ which turns $G$
into a topological group will be called called a group-topology.

Let $A$ be a set, $G$ be a topological group, and $\left( G_{\alpha
},\varphi _{\alpha }\right) _{\alpha \in A}$ be a family of pairs of
topological groups and morphisms $G_{\alpha }\rightarrow G$ such that $%
G=\tbigcup\nolimits_{\alpha \in A}\varphi _{\alpha }\left( G_{\alpha
}\right) .$\ \ Then $G$ is the inductive limit of the family $\left(
G_{\alpha }\right) $\ relative to the group-maps $\varphi _{\alpha }$\ in $%
\mathbf{Grp}$ (written $G=\lim\limits_{\longrightarrow }G_{\alpha }$ $\left[
\varphi _{\alpha },\mathbf{Grp}\right] $). \ Let $\mathfrak{T}_{\alpha }$ be
the topology of $G_{\alpha }.$ \ The final group-topology $\mathfrak{T}$\ of
the family $\left( \mathfrak{T}_{\alpha }\right) $ relative to the
group-maps $\varphi _{\alpha }$ is defined to be the finest \emph{%
group-topology} which makes all maps $\varphi _{\alpha }$ continuous. \ This
group topology exists, i.e. inductive limits exist in $\mathbf{TGr}$. \
Nevertheless, this intricate group-topology is not Hausdorff in general,
even if the topologies of the topological groups $G_{\alpha }$ are Hausdorff
(\cite{Eilenberg-Steenrod}, p. 132) and Abelian \cite{Baker}.

\begin{claim}
The following conditions are equivalent:\newline
(i) $G$ is endowed with the above final group-topology $\mathfrak{T}$ (the
topological group obtained in this way is written $G\left[ \mathfrak{T}%
\right] ,$ or $G$ when there is no confusion).\newline
(ii) $G\left[ \mathfrak{T}\right] $ is the \emph{inductive limit} of the
family $\left( G_{\alpha }\right) $\ relative to the group-maps $\varphi
_{\alpha }$\ in $\mathbf{TGrp}$ (written $G\left[ \mathfrak{T}\right]
=\lim\limits_{\longrightarrow }G_{\alpha }$ $\left[ \varphi _{\alpha },%
\mathbf{TGrp}\right] $), i.e. a map $\rho :G\rightarrow H,$ where $H$ is any
topological group, is a morphism in $\mathbf{TGr}$ if, and only if for each $%
\alpha \in A,$ $\rho \circ \varphi _{\alpha }:G_{\alpha }\rightarrow H$ is a
morphism.
\end{claim}

Since the groups $\varphi _{\alpha }\left( G_{\alpha }\right) $ and $%
G_{\alpha }/\ker \varphi _{\alpha }$ are canonically isomorphic, they are
identified; then the group-map $\varphi _{\alpha }$ and the canonical
surjection $\pi _{\alpha }:G_{\alpha }\twoheadrightarrow G_{\alpha }/\ker
\varphi _{\alpha }$\ are identified too. \ The final group-topology $%
\mathfrak{T}_{\alpha }^{q}$ on $G_{\alpha }/\ker \varphi _{\alpha }$
relative to $\pi _{\alpha }$ (resp. the final group-topology $\mathfrak{T}%
_{\alpha }^{f}$ on $\varphi _{\alpha }\left( G_{\alpha }\right) $ relative
to $\varphi _{\alpha }$) is the finest group-topology on $G_{\alpha }/\ker
\varphi _{\alpha }$ (resp. $\varphi _{\alpha }\left( G_{\alpha }\right) $)
which makes $\pi _{\alpha }$ (resp. $\varphi _{\alpha }$) continuous. \
Obviously, $\mathfrak{T}_{\alpha }^{q}$ and $\mathfrak{T}_{\alpha }^{f}$
coincide and $G=\lim\limits_{\longrightarrow }\varphi _{\alpha }\left(
G_{\alpha }\right) $ $\left[ j_{\alpha },\mathbf{TGr}\right] $ where $%
j_{\alpha }:\varphi _{\alpha }\left( G_{\alpha }\right) \hookrightarrow G$
is the inclusion.

\begin{conclusion}
(i) Let $A$ be right-filtrant and let $\left( G_{\alpha },\varphi _{\beta
}^{\alpha }\right) $ be a direct system in $\mathbf{TGr}.$ \ Let $%
G=\lim\limits_{\longrightarrow }G_{\alpha }$ $\left[ \varphi _{\alpha },%
\mathbf{TGr}\right] $ be the filtrant inductive limit of this system and $%
\alpha \preceq \beta $. \ The compatibility relation $\varphi _{\alpha
}=\varphi _{\beta }\circ \varphi _{\beta }^{\alpha }$ implies $\varphi
_{\alpha }\left( G_{\alpha }\right) \subset \varphi _{\beta }\left( G_{\beta
}\right) $. \ The group-map $j_{\beta }^{\alpha }:\varphi _{\alpha }\left(
G_{\alpha }\right) \rightarrow \varphi _{\beta }\left( G_{\beta }\right) $
induced by $\varphi _{\beta }^{\alpha }$ (for $\alpha \preceq \beta $) is
injective and continuous and is called the canonical monomorphism. \
Therefore the topological group $G$ is the filtrant inductive limit of the
direct system $\left( \varphi _{\alpha }\left( G_{\alpha }\right) ,j_{\beta
}^{\alpha }\right) .$\newline
(ii) An inductive limit of topological groups is called \emph{countable} if
the set of indices $A$ is countable.\newline
(iii) An inductive limit of topological groups is called \emph{strict} if it
is filtrant, $\varphi _{\beta }^{\alpha }$ is the inclusion $G_{\alpha
}\hookrightarrow G_{\beta },$ $\varphi _{\alpha }$ is the inclusion $%
G_{\alpha }\hookrightarrow G,$ and the topology of $\mathfrak{T}_{\alpha }$
and that induced in $G_{\alpha }$ by $\mathfrak{T}_{\beta }$ coincide
whenever $\alpha \preceq \beta .$
\end{conclusion}

\begin{lemma}
\label{lemma-inductive-limit-top-groups}Let $N$ be a normal subgroup of $%
G=\lim\limits_{\longrightarrow }G_{\alpha }$ $\left[ \varphi _{\alpha },%
\mathbf{TGr}\right] ,$ let $N_{\alpha }=\varphi _{\alpha }^{-1}\left(
N\right) $ and let $\bar{\varphi}_{\alpha }:G_{\alpha }/N_{\alpha
}\rightarrow G/N$ be the induced map. \ (i) Then%
\begin{equation*}
\lim\limits_{\longrightarrow }G_{\alpha }/N_{\alpha }\text{\ }\left[ \bar{%
\varphi}_{\alpha },\mathbf{TGr}\right] =G/N.
\end{equation*}%
(ii) Considering in $\mathbf{TGr}$ a direct system $\left( G_{\alpha
},\varphi _{\beta }^{\alpha }\right) _{\alpha \preceq \beta ,\alpha ,\beta
\in A}$ and the filtrant inductive limit $G=\lim\limits_{\longrightarrow
}G_{\alpha },$ the map $\bar{\varphi}_{\beta }^{\alpha }:F_{\alpha
}/N_{\alpha }\rightarrow F_{\beta }/N_{\beta }$ induced by $\varphi _{\beta
}^{\alpha }$ exists and is a morphism. \ In addition, $\left( G_{\alpha
}/N_{\alpha },\bar{\varphi}_{\beta }^{\alpha }\right) _{\alpha \preceq \beta
,\alpha ,\beta \in A}$ is a direct system in $\mathbf{TGr}$ whose filtrant
inductive limit\ is $G/N$.
\end{lemma}

\begin{proof}
(i) Consider the commutative diagram%
\begin{equation*}
\begin{array}{ccc}
G_{\alpha } & \overset{\varphi _{\alpha }}{\rightarrow } & G \\ 
\pi _{\alpha }\downarrow &  & \downarrow \pi \\ 
G_{\alpha }/N_{\alpha } & \overset{\bar{\varphi}_{\alpha }}{\rightarrow } & 
G/N%
\end{array}%
\end{equation*}%
where $\pi _{\alpha }:G_{\alpha }\twoheadrightarrow G_{\alpha }/N_{\alpha }$
and $\pi :G\twoheadrightarrow G/N$ are the canonical epimorphisms. \ We have%
\begin{equation*}
G/N=\lim\limits_{\longrightarrow }G\text{\ }\left[ \pi ,\mathbf{TGr}\right]
=\lim\limits_{\longrightarrow }\left( \lim\limits_{\longrightarrow
}G_{\alpha }\text{\ }\left[ \varphi _{\alpha },\mathbf{TGr}\right] \right) 
\text{\ }\left[ \pi ,\mathbf{TGr}\right] ,
\end{equation*}%
therefore (\cite{Bourbaki-E}, CST19, p. IV.20)%
\begin{eqnarray*}
G/N &=&\lim\limits_{\longrightarrow }G_{\alpha }\text{\ }\left[ \pi \circ
\varphi _{\alpha },\mathbf{TGr}\right] =\lim\limits_{\longrightarrow
}G_{\alpha }\text{\ }\left[ \bar{\varphi}_{\alpha }\circ \pi _{\alpha },%
\mathbf{TGr}\right] \\
&=&\lim\limits_{\longrightarrow }G_{\alpha }/N_{\alpha }\text{\ }\left[ \bar{%
\varphi}_{\alpha },\mathbf{TGr}\right] .
\end{eqnarray*}%
(ii) We know that $\varphi _{\alpha }=\varphi _{\beta }\circ \varphi _{\beta
}^{\alpha }$ whenever $\alpha \preceq \beta .$\ \ Let $x\in N_{\alpha },$
i.e. $\varphi _{\alpha }\left( x\right) \in N;$ equivalently, $\varphi
_{\beta }\left( \varphi _{\beta }^{\alpha }\left( x\right) \right) \in N,$
i.e. $\varphi _{\beta }^{\alpha }\left( x\right) \in N_{\beta }.$ \
Therefore, $\varphi _{\beta }^{\alpha }\left( N_{\alpha }\right) \subset
N_{\beta },$ and the map $\bar{\varphi}_{\beta }^{\alpha }:F_{\alpha
}/N_{\alpha }\rightarrow F_{\beta }/N_{\beta }$ induced by $\varphi _{\beta
}^{\alpha }$ exists. \ This map is easily seen to be continuous, and $\left(
F_{\alpha }/N_{\alpha },\bar{\varphi}_{\beta }^{\alpha }\right) _{\alpha
\preceq \beta ,\alpha ,\beta \in A}$ is a direct system in $\mathbf{TGr}$. \
By (i), the filtrant inductive limit of that direct system is $G/N$.
\end{proof}

\subsubsection{Inductive limits of locally convex spaces}

The above definitions, rationales and results still hold, \emph{mutatis
mutandis}, in the category $\mathbf{LCS}$. \ Final topologies hence
inductive limits exist in that category. \ Moreover -- contrary to what
happens in $\mathbf{TGrp}$ -- these topologies are easily described (\cite%
{Bourbaki-EVT}, p. II.29). \ The following is classical:

\begin{proposition}
Assuming that $F=\lim\limits_{\longrightarrow }F_{n}$ is a \emph{strict
countable inductive limit}, then (i) the topology induced in $F_{n}$ by that
of $F$ coincides with that of $F$ (so that the topology of $F$ is Hausdorff
if the topologies of the $F_{n}$ is Hausdorff), (ii) if $F_{n}$ is closed in 
$F_{n+1}$ for every $n,$ then $F_{n}$ is closed in $F,$ (iii) if $F_{n}$ is
complete\ for every $n,$ then $F$ is complete (\cite{Bourbaki-EVT}, \S II.4,
Prop. 9). \ In addition (\cite{Bourbaki-EVT}, \S II.4, Exerc. 14) the
topology of $F$ is the finest topology among all topologies compatible with
the vector space structure of $F$ (locally convex or not) which induce in $%
F_{n}$ a coarser topology than the given topology $\mathfrak{T}_{n}.$
\end{proposition}

In what follows, we will say that the strict inductive limit $%
F=\lim\limits_{\longrightarrow }F_{n}$ is nontrivial if $F_{n}\subsetneq
F_{n+1}.$

\subsubsection{Locally convex hulls of locally convex spaces}

Let $\left( F_{\alpha }\right) _{\alpha \in A}$ be a family of LCS's, $F$ be
an LCS, and $\left( \varphi _{\alpha }\right) _{\alpha \in A}$ be a family
of linear maps $F_{\alpha }\rightarrow F.$ \ Then, algebraically,%
\begin{equation}
E:=\sum_{\alpha \in A}\varphi _{\alpha }\left( F_{\alpha }\right) \cong
\left( \bigoplus\nolimits_{\alpha \in A}F_{\alpha }\right) /H
\label{isom-locally-convex-hull}
\end{equation}%
where $H=\ker \left( \varphi \right) ,$ $\varphi :\left( x_{\alpha }\right)
_{\alpha \in A}\in \bigoplus\nolimits_{\alpha \in A}F_{\alpha }\rightarrow
\sum_{\alpha }\varphi _{\alpha }\left( F_{\alpha }\right) .$ \ The hull
topology of $E$ is defined to be the finest locally convex topology which
makes all maps $\varphi _{\alpha }$ continuous; $E,$ endowed with this
topology, is called the locally convex hull of $\left( F_{\alpha },\varphi
_{\alpha }\right) $ (\cite{KotheI}, p. 215). \ Then $\left( \ref%
{isom-locally-convex-hull}\right) $ is an isomorphism of LCS's, $H$ is a
subspace of $\bigoplus\nolimits_{\alpha \in A}F_{\alpha }$, and $H$ is
closed if, and only if $E$ is Hausdorff.

Conversely, let $\left( F_{\alpha }\right) _{\alpha \in A}$ be a family of
LCS's and $H$ be a subspace of $\bigoplus\nolimits_{\alpha \in A}F_{\alpha
}. $ \ Then $E=\left( \bigoplus\nolimits_{\alpha \in A}F_{\alpha }\right) /H$
is a locally convex hull of $\left( F_{\alpha },\varphi _{\alpha }\right) $
where $\varphi _{\alpha }=\varphi \left\vert F_{\alpha }\right. ,$ $\varphi
:\bigoplus\nolimits_{\alpha \in A}F_{\alpha }\twoheadrightarrow E$
(canonical map).

Assuming that $\left( F_{\alpha },\varphi _{\beta }^{\alpha }\right)
_{\alpha \in A}$ is a direct system and $\varphi _{\alpha }\left( F_{\alpha
}\right) \subset \varphi _{\beta }\left( F_{\beta }\right) $ for $\alpha
\preccurlyeq \beta ,$ the locally convex hull $\sum_{\alpha \in A}\varphi
_{\alpha }\left( F_{\alpha }\right) =\bigcup\nolimits_{\alpha \in A}\varphi
_{\alpha }\left( F_{\alpha }\right) $ coincides with $\lim\limits_{%
\longrightarrow }F_{\alpha }$. \ Since $\varphi _{\beta }^{\alpha }$ is
continuous, the canonical injection $j_{\beta }^{\alpha }:\varphi _{\alpha
}\left( F_{\alpha }\right) \rightarrow \varphi _{\beta }\left( F_{\beta
}\right) $ is continuous. \ Therefore, a locally convex hull is a
generalization of an inductive limit in the category $\mathbf{LCS}$.

\section{Properties $\mathfrak{G}_{\mathcal{C}}$, $\mathfrak{O}_{\mathcal{C}%
} $ and their relations\label{sect-main-thm}}

Let $E,F$ be topological spaces and $\mathcal{C}\subset \mathfrak{P}\left(
E\times F\right) $ be such that

\begin{equation*}
\left\{ \text{closed subsets of }E\times F\right\} \subset \mathcal{C}
\end{equation*}
\ For example, $\mathcal{C}$ can be the set of all closed, or of all
sequentially closed, or of all Borel subsets of $E\times F.$

Let $\mathbf{TCat}$ be any of the categories $\mathbf{TGrp}$, $\mathbf{TVsp}%
_{\mathbf{k}}$ and $\mathbf{LCS}$; then the notion of strict morphism is
well-defined in $\mathbf{TCat}$ (\cite{Bourbaki-TG-I}, p. III.16, Def. 1). \
A morphism $f:E\rightarrow F,$ where $E,F\in \mathbf{TCat}$, is \emph{strict}
if, and only if for any open subset $\Omega $ of $E,$ $f\left( \Omega
\right) $ is open in $f\left( E\right) .$ \ Therefore, if $f$ is a
surjective morphism, $f$ is strict if, and only if it is an open mapping.

\begin{definition}
$\mathfrak{G}_{\mathcal{C}}$ and $\mathfrak{O}_{\mathcal{C}}$ are the
following properties of pairs of objects of $\mathbf{TCat}$:\newline
(i) $\left( E,F\right) $ is $\mathfrak{G}_{\mathcal{C}}$ if whenever a map $%
u:E\rightarrow F$ is such that $\limfunc{Gr}\left( u\right) \in \mathcal{C}$%
, then $u$ is a morphism.\newline
(ii) $\left( E,F\right) $ is $\mathfrak{O}_{\mathcal{C}}$ if whenever $G\in 
\mathbf{TCat}\cap \mathcal{C},G\subset E\times F,$ and $v:G%
\twoheadrightarrow E$ is a surjective morphism, then $v$ is strict.
\end{definition}

\begin{example}
Let $\mathbf{TCat}=\mathbf{TVsp}_{\mathbf{k}}$, let $E,F$ be two metrizable
complete TVS's and let $\mathcal{C}$ be the set of all closed subsets of $%
E\times F.$\newline
(1) Let $u:E\rightarrow F$ be a map. If $\limfunc{Gr}\left( u\right) \in 
\mathcal{C},$ $u$ is a morphism by the usual Closed Graph Theorem, thus $%
\left( E,F\right) $ is $\mathfrak{G}_{\mathcal{C}}.$\newline
(2) Every space $G\in \mathcal{C}$ is metrizable and complete, thus every
surjective morphism $v:G\twoheadrightarrow E$ is strict by the usual Open
Mapping Theorem, thus $\left( E,F\right) $ is $\mathfrak{O}_{\mathcal{C}}.$
\end{example}

\begin{lemma}
\label{lemma-closed-graph-isomorphism}If $E$ is Hausdorff, $\left(
E,F\right) $ is $\mathfrak{G}_{\mathcal{C}},$ and $v:F\rightarrow E$ is a
bijective morphism, then $v$ is an isomorphism.
\end{lemma}

\begin{proof}
The graph of $v$ is closed (Lemma \ref{lemma-continuous-then-closed}), thus
so is the graph of $v^{-1}:E\rightarrow F$. \ Therefore, this graph is $%
\mathcal{C}$ in $E\times F$, and since $\left( E,F\right) $ is $\mathfrak{G}%
_{\mathcal{C}}\mathfrak{,}$ $v^{-1}$ is continuous.
\end{proof}

The relations between $\mathfrak{G}_{\mathcal{C}}$ and $\mathfrak{O}_{%
\mathcal{C}}$ are expressed in the following theorem when $\mathbf{TCat}=%
\mathbf{TGrp}$:

\begin{theorem}
\label{Main thm}Let $E,F$ be topological groups.\newline
(i) $\mathfrak{O}_{\mathcal{C}}\Rightarrow \mathfrak{G}_{\mathcal{C}}.$%
\newline
(ii) Conversely, if $E$ is Hausdorff and $\left( E,F/N\right) $ is $%
\mathfrak{G}_{\mathcal{C}}$ for every closed normal subgroup $N$ of $F$,
then every surjective morphism $v:F\twoheadrightarrow E$ is strict. \ 
\newline
(iii) Therefore whenever a morphism $v:F\rightarrow E$ is such that $\left(
E,\func{im}\left( v\right) /\left( N\cap \func{im}\left( v\right) \right)
\right) $ is $\mathfrak{G}_{\mathcal{C}}$ for every normal subgroup $N$ of $%
F $ such that $N\cap \func{im}\left( v\right) $ is closed in $\func{im}%
\left( v\right) ,$ this morphism $v$ is strict.
\end{theorem}

\begin{proof}
(i): This proof is a variant of that of (\cite{Treves}, Prop. 17.3). \ Let $%
u:E\rightarrow F$ be a map (i.e. a morphism of groups). \ Its graph $%
\limfunc{Gr}\left( u\right) $ is a subgroup $G$ of $E\times F.$ Let us endow 
$E\times F$ with the product topology $\pi $ and $G$ with the topology
induced by $\pi .$ \ Then the two projections 
\begin{equation*}
p:G\rightarrow E:\left( x,u\left( x\right) \right) \mapsto x,\quad
q:G\rightarrow F:\left( x,u\left( x\right) \right) \mapsto u\left( x\right)
\end{equation*}%
are continuous and%
\begin{equation*}
u=q\circ p^{-1}.
\end{equation*}

Assume that $\left( E,F\right) $ is $\mathfrak{O}_{\mathcal{C}}.$ \ Then, if 
$\limfunc{Gr}\left( u\right) \in \mathcal{C}$ , $p$ is open, thus $p^{-1}$
is continuous, and $u$ is continuous too. \ Therefore $\left( E,F\right) $
is $\mathfrak{G}_{\mathcal{C}}.$

(ii): Let $v:F\twoheadrightarrow E$\ be a surjective morphism and consider
the following commutative diagram%
\begin{equation*}
\begin{array}{ccc}
F & \overset{v}{\longrightarrow } & E \\ 
\downarrow \pi & \nearrow \bar{v} &  \\ 
F/\ker \left( v\right) &  & 
\end{array}%
\end{equation*}%
where $\bar{v}:F/\ker \left( v\right) \rightarrow E$ is a bijective morphism.

Assume that $E$ is Hausdorff. \ Then the normal subgroup $N=\ker \left(
v\right) $ of $F$\ is closed. \ If in addition $\left( E,F/N\right) $ is $%
\mathfrak{G}_{\mathcal{C}}\mathfrak{,}$ $\bar{v}$ is an isomorphism by Lemma %
\ref{lemma-closed-graph-isomorphism}. \ As a result, the morphism $v$ is
strict.

(iii) The subgroup $N^{\prime }=N\cap \func{im}\left( v\right) $ is normal
in $F^{\prime }=\func{im}\left( v\right) $, thus (iii) is a consequence of
(ii).
\end{proof}

\section{Typology of topological spaces, groups and vector spaces\label%
{Sect-typology}}

\subsection{Typology of topological spaces}

A subspace $A$ of a topological space $X$ is \emph{rare} when $X\backslash A$
is everywhere dense (in other words, $\overset{\circ }{\bar{A}}$ is empty).
\ A subspace $B$ of $X$ is \emph{meagre} (or "of first category") if it is a
countable union of rare subspaces. \ A subset of $X$ is called non-meagre
(or "of second category") when it is not meagre (i.e. when it is not of
first category). \ The following is classical (\cite{Bourbaki-TG-II}, p.
IX.54, Def. 3).

\begin{claim}
Let $X$ be a topological space.\newline
(1) The following conditions are equivalent:\newline
(a) every open subset of $X$ which is non-empty is non-meagre;\newline
(b) if a subset $B$ is meagre, then $X\backslash B$ is everywhere dense,
i.e. $\mathring{B}\ $is empty;\newline
(c) every countable union of closed sets with empty interior has an empty
interior;\newline
(c') every countable intersection of open everywhere dense subsets in $X$ is
everywhere dense in $X.$\newline
(2) A topological space $X$ is \emph{Baire} if the above equivalent
conditions hold.
\end{claim}

A complete metrizable topological space is Baire, and a locally compact
topological space as well ("Baire's theorem").

\begin{definition}
A topological space $X$ is\newline
$\bullet $ \emph{first countable} (or satisfies the \emph{first axiom of
countability}) if every point has a countable base of neighborhoods;\newline
$\bullet $ \emph{second countable} (or satisfies the \emph{second axiom of
countability}) if its topology has a countable base;\newline
$\bullet $ \emph{separable} if there exists a countable dense subspace.
\end{definition}

In general,%
\begin{equation*}
\text{second countable }\Rightarrow \text{ first countable and separable}
\end{equation*}

Let $X$ be a topological space and $\mathfrak{P}\left( X\right) $ (resp. $%
\mathfrak{K}\left( X\right) $) be the set of all subsets (resp. of all
compact subsets) of $X.$ \ The topological space $X$ is

\begin{itemize}
\item \emph{completely regular} if it is uniformizable and Hausdorff (\cite%
{Bourbaki-TG-II}, p. IX.8, Def. 4),

\item \emph{inexhaustible} if for every countable family of closed subsets $%
A_{n}\subset X,$ the relation $X=\cup _{n}A_{n}$ implies that there exists $%
n $ such that $A_{n}$ has nonempty interior (\cite{Hofmann}, Def. 3.3.a).

\item \emph{Polish} if it is metrizable, complete and separable (\cite%
{Bourbaki-TG-II}, p. IX.57, Def. 1),

\item \emph{Suslin} if it is Hausdorff and there exist a Polish space $P$
and a continuous surjection $P\rightarrow X$ (\cite{Bourbaki-TG-II}, p.
IX.59, Def. 1),

\item \emph{Lindel\"{o}f} if, from any open covering, one can extract a
countable covering (\cite{Bourbaki-TG-II}, p. IX.76, Def. 1),

\item $K_{\sigma \delta }$ if it is a countable intersection of countable
unions of compact sets (\cite{Choquet}, \S 3),

\item $K$\emph{-analytic} if there exists a $K_{\sigma \delta }$ space $Y$
and a continuous surjection $f:Y\rightarrow X.$
\end{itemize}

The definitions given below are due to Martineau \cite{Martineau-Studia} ($K$%
-Suslin spaces) and Valdivia (\cite{Valdivia-book}, p. 52) (quasi-Suslin
spaces).

\begin{definition}
\label{def-K-Suslin-quasi-Suslin}A topological space $X$ is called\newline
(1) $K$\emph{-Suslin} if there exist a Polish space $P$ and a map $%
T:P\rightarrow \mathfrak{K}\left( X\right) $ such that\newline
(i) $X=\bigcup\nolimits_{p\in P}T\left( p\right) ,$ and\newline
(ii) $T$ is semi-continuous, i.e. for every $p\in P,$ and for every
neighborhood $V$ of $T\left( p\right) $ in $X,$ there exists a neighborhood $%
U$ of $p$ in $P$ such that $T\left( U\right) \subset V;$\newline
(2) \emph{quasi-Suslin} if there exist a Polish space $P$ and a map $%
T:P\rightarrow \mathfrak{P}\left( X\right) $ such that\newline
(i') $X=\bigcup\nolimits_{p\in P}T\left( p\right) ,$\newline
(ii') if $\left( p_{n}\right) $ is a sequence in $P$ converging to $p$ and
if $x_{n}\in T\left( p_{n}\right) $ for every $n\geq 0,$ then $\left(
x_{n}\right) $ has an cluster point in $X$ belonging to $T\left( p\right) .$
\end{definition}

A completely regular topological space is $K$-Suslin if, and only if it is $%
K $-analytic \cite{Rogers}.

For a completely regular topological space,%
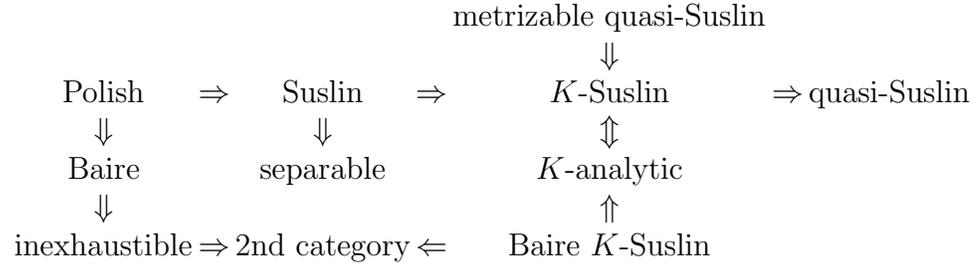
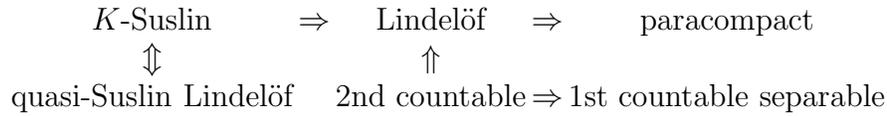
\begin{equation*}
\begin{array}{ccccccc}
&  &  &  & \text{metrizable quasi-Suslin} &  &  \\ 
&  &  &  & \Downarrow &  &  \\ 
\text{Polish} & \Rightarrow & \text{Suslin} & \Rightarrow & K\text{-Suslin}
& \Rightarrow & \text{quasi-Suslin} \\ 
\Downarrow &  & \Downarrow &  & \Updownarrow &  &  \\ 
\text{Baire} &  & \text{separable} &  & K\text{-analytic} &  &  \\ 
\Downarrow &  &  &  & \Uparrow &  &  \\ 
\text{inexhaustible} & \Rightarrow & \text{2nd category} & \Leftarrow & 
\text{Baire }K\text{-Suslin} &  & 
\end{array}%
\end{equation*}%
and%
\begin{equation*}
\begin{array}{ccccc}
K\text{-Suslin} & \Rightarrow & \text{Lindel\"{o}f} & \Rightarrow & \text{%
paracompact} \\ 
\Updownarrow &  & \Uparrow &  &  \\ 
\text{quasi-Suslin Lindel\"{o}f} &  & \text{2nd countable} & \Rightarrow & 
\text{1st countable separable}%
\end{array}%
\end{equation*}

\begin{proof}
Polish $\Rightarrow $ Suslin $\Rightarrow K$-Suslin $\Rightarrow $
quasi-Suslin is clear \cite{Martineau-Studia}, (\cite{Valdivia-book}, p. 60,
(3)). \ Metrizable quasi-Suslin $\Rightarrow K$-Suslin by (\cite%
{Valdivia-book}, p. 67, (26)). \ Polish $\Rightarrow $ Baire (Baire's
theorem) $\Rightarrow $ nonmeagre in itself $\Leftrightarrow $ inexhaustible
(\cite{Bourbaki-TG-II}, p. IX.112, Exerc. 7) $\Rightarrow $ of second
category. \ $K$-Suslin $\Rightarrow $ Lindel\"{o}f $\Rightarrow $
paracompact by (\cite{Martineau-Studia}, Prop. 2) and (\cite{Bourbaki-TG-II}%
, p. IX.76, Prop. 2). \ 2nd countable $\Rightarrow $ Lindel\"{o}f and 2nd
countable $\Rightarrow $ 1st countable separable according to the
definitions. A space is $K$-Suslin if, and only if it is quasi-Suslin and
Lindel\"{o}f (\cite{Orihuela}, p. 148).
\end{proof}

In addition, every space which is metrizable, locally compact and countable
at infinity is Polish.

A subset $S$ of a topological space is called \emph{Borel} (resp. \emph{%
sq-Borel}) if it belongs to the smallest $\sigma $-algebra $\mathfrak{B}$ of 
$X$ generated by the closed (resp. sequentially closed) subsets of $X$ (\cite%
{De Wilde}, p. 130, Def. 7.2.6). \ (A Borel set is sq-Borel.)

Let $S$ be a subset of a topological space $X$; $S$ is called $F$\emph{-Borel%
} if it belongs to the smallest subset $F\mathfrak{B}$ of $\mathfrak{P}%
\left( X\right) $ which is stable by countable union, countable intersection
and which contains the closed subsets of $X$ \cite{Martineau-Studia}. \ Note
that $F\mathfrak{B}$ is not a $\sigma $-algebra in general, for $A\in F%
\mathfrak{B}$ does not imply $X\backslash A\in F\mathfrak{B}$ (in
particular, open subsets of $X$ need not belong to $F\mathfrak{B}$). \
Obviously, $F\mathfrak{B}\subset \mathfrak{B}.$

Considering Hausdorff spaces, the following hereditary properties hold (\cite%
{Bourbaki-TG-II},\S IX.6), (\cite{Valdivia-book}, \S I.4):

\begin{itemize}
\item An open subset of a Polish (resp. Suslin, Baire) space is such.

\item A closed subspace of a Suslin (resp. $K$-Suslin, quasi-Suslin) space
is such. \ Moreover, a sequentially closed subspace of a Suslin space is
Suslin.

\item A sequentially closed subspace and the image by a linear continuous
surjection of a semi-Suslin space are again semi-Suslin.

\item A countable sum or product of Polish (resp. Suslin) spaces is again
such.

\item A countable union of Suslin spaces is again Suslin.

\item A finite product of $K$-Suslin (resp. quasi-Suslin) spaces is again
such.

\item A countable product, a countable intersection, a countable projective
limit of\ Suslin (resp. $K$-Suslin, quasi-Suslin, semi-Suslin) spaces is
such.

\item A countable intersection of Polish spaces is again Polish.

\item A product of metrizable complete spaces is a Baire space (but a
product of Baire spaces need not be Baire).

\item The image by a continuous surjection of a Suslin (resp. $K$-Suslin,
quasi-Suslin, Lindel\"{o}f) space is such. \ Moreover the image of a Suslin
space by a sequentially continuous surjection is Suslin.

\item If $\left( X_{n}\right) $ is a sequence of $K$-Suslin (resp.
quasi-Suslin) spaces covering $X,$ then $X$ is $K$-Suslin (resp.
quasi-Suslin).

\item sq-Borel subspaces of Suslin spaces are Suslin (\cite{De Wilde}, p.
130, Prop. VII.2.7).

\item A space which contains a dense Baire subspace is Baire \cite%
{Haworth-McCoy}.

\item $F$-Borel subspaces of $K$-Suslin spaces are $K$-Suslin \cite%
{Martineau-Studia}.
\end{itemize}

\subsection{Typology of topological groups}

A topological group is uniformizable, thus a Hausdorff topological group is
completely regular. \ A topological group $G$ is metrizable if, and only if
it is Hausdorff and first countable (\cite{Bourbaki-TG-II}, p. IX.23, Prop.
1).

\begin{lemma}
A topological group is second countable if, and only if it is both first
countable and separable (\cite{Hofmann}, Exerc. E3.3).
\end{lemma}

Consider a Hausdorff topological group $G$ and a subgroup $H.$ \ The
homogeneous space $G/H$ is Hausdorff if, and only if $H$ a closed subspace
of $G.$ \ Since $G/H$ is the image of the continuous canonical map $%
G\rightarrow G/H$ we see that

\begin{itemize}
\item The quotient of a Suslin (resp. $K$-Suslin, quasi-Suslin, Lindel\"{o}%
f) group by a closed subgroup is Suslin (resp. $K$-Suslin, quasi-Suslin,
Lindel\"{o}f).
\end{itemize}

Let $\left( G_{\alpha }\right) $ be a family of topological groups, and $%
G=\lim\limits_{\longrightarrow }G_{\alpha }$ in $\mathbf{TGr}.$ \ The
topological group $G$ is called

\begin{itemize}
\item of type $\beta $ is the groups $G_{\alpha }$ are Suslin and Baire.

\item of type $K.\beta $ if the groups $G_{\alpha }$ are $K$-Suslin and
Baire \cite{Martineau-Studia}.
\end{itemize}

A metrizable separable group is second countable, thus Lindel\"{o}f. \ If in
addition it is locally compact and countable at infinity, it is Polish.

\subsection{Typology of locally convex spaces}

Let $E$ be a TVS and $A$ be a subset of $E$. \ This set is called \emph{%
CS-compact} if every series of the form%
\begin{equation*}
\dsum\limits_{n\geq 0}a_{n}x_{n},\quad x_{n}\in A,\quad a_{n}\geq 0,\quad
\dsum\limits_{n\geq 0}a_{n}=1
\end{equation*}%
converges to a point of $A$. \ If $E$ is an LCS, every CS-compact set is
bounded and convex. \ The set of all CS-compact subsets of $E$ is denoted by 
$\mathcal{S}\left( E\right) .$

The following definition is due to Valdivia (\cite{Valdivia-book}, p. 78):

\begin{definition}
A locally convex space $E$ is called \emph{semi-Suslin} if it is locally
convex and there exist a Polish space $P$ and a map $T:P\rightarrow \mathcal{%
S}\left( E\right) $ such that\newline
(i\textquotedblright ) $E=\bigcup\nolimits_{p\in P}T\left( p\right) ,$%
\newline
(ii\textquotedblright ) if $\left( p_{n}\right) $ is a sequence in $P$
converging to $p$, there exists $A\in \mathcal{S}\left( E\right) $ such that 
$T\left( p_{n}\right) \in A$ for all $n.$
\end{definition}

For an LCS (\cite{Valdivia-book}, p. 80, (3)) every quasi-Suslin and locally
complete is semi-Suslin. \ A Fr\'{e}chet space and the strong dual of a
metrizable locally convex space are semi-Suslin (\cite{Valdivia-book}, p.
81). \ Conversely, a convex-Baire semi-Suslin space is Fr\'{e}chet (\cite%
{Valdivia-book}, p. 92, (26)).

A Hausdorff LCS is called \emph{ultrabornological} if it is an inductive
limit of Banach spaces. \ A Fr\'{e}chet space, the strong dual of a
reflexive Fr\'{e}chet space, a quasi-complete bornological space, an
inductive limit of ultrabornological spaces, are all ultrabornological (\cite%
{Bourbaki-EVT}, n$%
%TCIMACRO{\U{b0}}%
%BeginExpansion
{{}^\circ}%
%EndExpansion
$III.6.1 and \S III.4, Exercise 20 \& 21; \cite{Grothendieck-EVT}, Chap. 4,
Part 1, \S 5). \ A $\beta $ LCS is ultrabornological, but not conversely.

A $K.\beta $ LCS is called $K$-ultrabornological.

A barrel in a TVS is a set which is convex, balanced, absorbing and closed.
\ A barrelled space is an LCS in which every barrel is a neighborhood of $0$
(\cite{Bourbaki-EVT}, \S III.4, Def. 2). \ An inductive limit and a product
of barrelled spaces, a quotient of a barrelled space, are all barrelled (%
\cite{Bourbaki-EVT}, p. III.25). \ A Banach space is barrelled, therefore an
ultrabornological space is barrelled. \ A nontrivial $\left( \mathcal{LF}%
\right) $-space (i.e. a nontrivial strict inductive limit of Fr\'{e}chet
spaces) is ultrabornological but not Baire (assuming that $E=\cup
_{n}E_{n},E_{n}\subsetneq E_{n+1}$ and $E_{n}$ is closed in $E,$ then $E_{n}=%
\bar{E}_{n}$ and $\mathring{E}_{n}=\varnothing $ because $E_{n}\subsetneq E$
is not a neighborhood of $0$ in $E;$ thus $E_{n}$\ is rare and $E$ is not
Baire). \ Thus an ultrabornological space need not be Baire.

\begin{claim}
(\cite{Valdivia-book}, \S I.2.2). \ Let $E$ be an LCS. \ (1) The following
conditions are equivalent:$\newline
$(a) Given any sequence $\left( A_{n}\right) $ of closed convex subsets with
empty interior, $\cup _{n}A_{n}$ has empty interior.\newline
(b) If $\left( A_{n}\right) $ is a covering of $E$ by closed convex sets,
there exists an index $n$ such that $\mathring{A}_{n}$ is nonempty.\newline
(2) The space $E$ is called \emph{convex-Baire} if these equivalent
conditions hold.
\end{claim}

Every convex-Baire space is barrelled (\cite{Valdivia-book}, p. 94).

The notion of webbed space is due to De Wilde \cite{De Wilde}. \ 

\begin{definition}
Let $E$ be an LCS, $\mathcal{W}=\left\{ C_{n_{1},...,n_{k}}\right\} $ be a
collection of subsets $C_{n_{1},...,n_{k}}$ of $E$ where $k$ and the $%
n_{1},...,n_{k}$ run through the natural integers. $\ \mathcal{W}$ is called
a \emph{web} if%
\begin{equation*}
E=\bigcup\nolimits_{n_{1}=1}^{\infty }C_{n_{1}},\quad
C_{n_{1},...,n_{k-1}}=\bigcup\nolimits_{n_{k}=1}^{\infty }C_{n_{1},...,n_{k}}
\end{equation*}%
and a $\mathcal{C}$\emph{-web} if in addition for every sequence $\left(
n_{k}\right) $ there exists a sequence of positive numbers $\rho _{k}$ such
that for all $\lambda _{k}$ such that $0\leq \lambda _{k}\leq \rho _{k},$
and all $x_{k}\in C_{n_{1},...,n_{k}},$ the series $\sum_{k=1}^{\infty
}\lambda _{k}x_{k}$ converges in $E.$\newline
$E$ is called a \emph{webbed space} if there exists a $\mathcal{C}$-web in $%
E $.
\end{definition}

A Fr\'{e}chet space, a sequentially complete $\left( \mathcal{DF}\right) $
space (in the sense of Grothendieck \cite{Grothendieck-EVT}), the strong or
weak dual of a countable inductive limit of metrizable locally convex
spaces, are webbed (\cite{KotheII}, \S 35.4), (\cite{De Wilde}, p. 57, Prop.
IV.3.3). \ In addition, an LCS space is Fr\'{e}chet if, and only if it is
both webbed and Baire (\cite{Jarchow}, 5.4.4).

Barrelled spaces, convex-Baire spaces, $\beta $ spaces, ultrabornological
spaces and webbed spaces enjoy the following hereditary properties:

\begin{itemize}
\item A quotient of a barrelled (resp. semi-Suslin, webbed) space is
barrelled (resp. semi-Suslin, webbed), and a Hausdorff quotient of a
convex-Baire space is convex-Baire.

\item The image of a semi-Suslin (resp. webbed) space by a continuous map is
semi-Suslin (resp. webbed) (\cite{Valdivia-book}, p. 79), (\cite{KotheII},
p. 61, (2)). \ Since a Suslin locally convex space is the continuous image
of a separable Fr\'{e}chet space, and since a Fr\'{e}chet space is webbed, a
Suslin space is webbed.

\item A product of barrelled (resp. convex-Baire) spaces is barrelled (resp.
convex-Baire). \ A countable product of ultrabornological (resp. webbed)
spaces is ultrabornological (resp. webbed). \ A finite product of
semi-Suslin spaces is semi-Suslin (\cite{Valdivia-book}, p. 80).

\item An inductive limit of barrelled (resp. $\beta $, ultrabornological)
spaces is barrelled (resp. $\beta $, ultrabornological). \ A locally convex
hull of barrelled spaces is barrelled (\cite{KotheI}, p. 368, (3)). \ An
inductive limit and a locally convex hull of countably many webbed spaces is
webbed. \ A countable inductive limit of Suslin spaces is Suslin (since it
is the countable union of continuous images of Suslin spaces).

\item A dense subspace of a convex-Baire space is convex-Baire (therefore a
convex-Baire space need not be complete).

\item A sequentially closed subspace of a semi-Suslin space is semi-Suslin (%
\cite{Valdivia-book}, p. 79).
\end{itemize}

Let $E$ be an LCS and $\mathfrak{T}$ its topology. \ Then $\mathfrak{T}^{f}$
is defined to be the finest topology on $E^{\prime }$ which coincides with
the topology induced by the weak topology $\sigma \left( E^{\prime
},E\right) $ on the $\mathfrak{T}$-equicontinuous subsets of $E^{\prime }.$
\ Therefore, a linear subspace $Q\subset E^{\prime }$ is $\mathfrak{T}^{f}$%
-closed if, and only if for each equicontinuous subset $A\subset E^{\prime
}, $ $Q\cap A$ is $\sigma \left( E^{\prime },E\right) $-closed in $A$
(equivalently, $Q$ is $\mathfrak{T}^{f}$-closed if, and only if for every
neighborhood $U$ of $0$ in $E,$ $Q\cap U^{0}$ is $\sigma \left( E^{\prime
},E\right) $-closed in $U^{0}$, where $U^{0}$ is the polar of $U$).

\begin{definition}
An LCS $E$ is called a \emph{Pt\'{a}k space} (or a $B$\emph{-complete space}%
) if every $\mathfrak{T}^{f}$-closed linear subspace of $E^{\prime }$ is $%
\sigma \left( E^{\prime },E\right) $-closed. \ An LCS $E$ is called an \emph{%
infra-Pt\'{a}k space} (or a $B_{r}$\emph{-complete space}) if every weakly
dense $\mathfrak{T}^{f}$-closed linear subspace of $E^{\prime }$ coincides
with $E^{\prime }.$
\end{definition}

A Pt\'{a}k space is infra-Pt\'{a}k space (\cite{Ptak}, (4.2)) but not
conversely \cite{Valdivia-84}, and an infra-Pt\'{a}k space is complete (\cite%
{KotheII}, \S 34.2).

A closed subspace of an infra-Pt\'{a}k (resp. a Pt\'{a}k) space is again
infra-Pt\'{a}k (resp. Pt\'{a}k) (\cite{Ptak}, (3.9); \cite{Horvath}, \S %
3.17).

A Fr\'{e}chet space, the strong dual of a reflexive Fr\'{e}chet space, and
the Hausdorff quotient of a Pt\'{a}k space,\ are all Pt\'{a}k spaces (\cite%
{Horvath}, \S 3.17). \ A product of two Pt\'{a}k spaces need not be Pt\'{a}k.

A separable Fr\'{e}chet space is Polish, thus $K$-Suslin, and Baire, hence $%
\beta $. \ Conversely, a locally convex space which is Baire and $K$-Suslin
is separable Fr\'{e}chet (\cite{Valdivia-book}, p. 64, (21)). \ If $E$ is a
Fr\'{e}chet space, its weak dual is $K$-Suslin (but not separable in
general, thus not Suslin), and $E^{\prime \prime }\left[ \sigma \left(
E^{\prime \prime },E^{\prime }\right) \right] $ is quasi-Suslin; if $E$ is
reflexive, $\left( E,\sigma \left( E,E^{\prime }\right) \right) $ is $K$%
-Suslin (\cite{Treves}, p. 557), (\cite{Valdivia-book}, p. 66, (24)).

The above is summarized below, where $\left( \mathcal{F}\right) $ means Fr%
\'{e}chet, $\left( s\mathcal{F}\right) $ means separable Fr\'{e}chet, $%
\left( c\mathcal{LF}\right) $ means a countable inductive limit (non
necessarily strict) of Fr\'{e}chet spaces, $\left( c\mathcal{L}s\mathcal{F}%
\right) $ means countable inductive limit of separable Fr\'{e}chet spaces,
and l.c.h. means locally convex hull:%
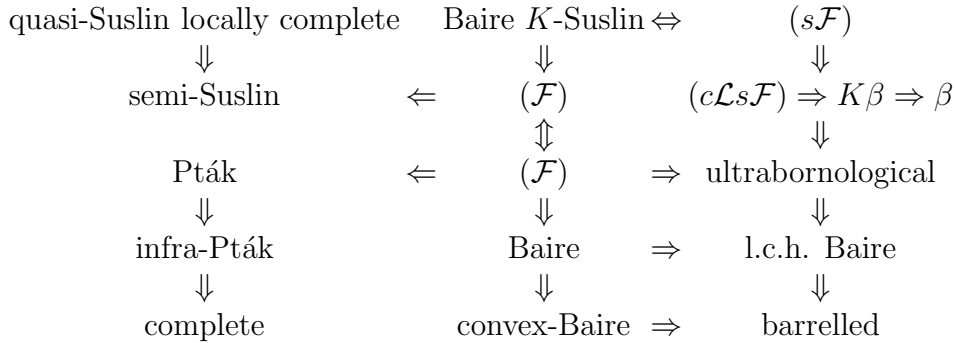
\begin{equation*}
\begin{array}{ccccc}
\text{quasi-Suslin locally complete} &  & \text{Baire }K\text{-Suslin} & 
\Leftrightarrow & \left( s\mathcal{F}\right) \\ 
\Downarrow &  & \Downarrow &  & \Downarrow \\ 
\text{semi-Suslin} & \Leftarrow & \left( \mathcal{F}\right) &  & \left( c%
\mathcal{L}s\mathcal{F}\right) \Rightarrow K\beta \Rightarrow \beta \\ 
&  & \Updownarrow &  & \Downarrow \\ 
\text{Pt\'{a}k} & \Leftarrow & \left( \mathcal{F}\right) & \Rightarrow & 
\text{ultrabornological} \\ 
\Downarrow &  & \Downarrow &  & \Downarrow \\ 
\text{infra-Pt\'{a}k} &  & \text{Baire} & \Rightarrow & \text{l.c.h. Baire}
\\ 
\Downarrow &  & \Downarrow &  & \Downarrow \\ 
\text{complete} &  & \text{convex-Baire} & \Rightarrow & \text{barrelled}%
\end{array}%
\end{equation*}

and%
\begin{equation*}
\begin{tabular}{lllll}
$\left( s\mathcal{F}\right) $ & $\Rightarrow $ & $\left( c\mathcal{L}s%
\mathcal{F}\right) $ & $\Rightarrow $ & Suslin \\ 
$\Downarrow $ &  &  &  & $\Downarrow $ \\ 
$\left( \mathcal{F}\right) $ & $\Leftrightarrow $ & webbed Baire & $%
\Leftarrow $ & webbed \\ 
$\Downarrow $ &  &  &  &  \\ 
ultrabornological & $\Rightarrow $ & l.c.h. metrizable Baire & $\Leftarrow $
& Baire \\ 
&  & $\Downarrow $ &  &  \\ 
&  & l.c.h. Baire & $\Rightarrow $ & barrelled%
\end{tabular}%
\end{equation*}

\section{Theorems in topological groups\label{Sect-thm-top-groups}}

Let us review three important versions of the (generalized) Closed Graph
Theorem and of the Open Mapping Theorem for topological groups. \ Let $G,H$
be Hausdorff topological groups and $u:G\rightarrow H$ be a map. \ The table
below summarizes the results. \ The first column is the statement number,
the last column contains the reference. \ In the other columns, the
conditions on $G,$ $H$ and $\limfunc{Gr}\left( u\right) $ under which $u$ is
a morphism are indicated.

\subsection{Closed Graph Theorems\label{subsect-closed-graph-groups}}

Let $u:G\rightarrow H$ be a map. \ Then $u$ is continuous under any of the
conditions on $G,H$ and $\limfunc{Gr}\left( u\right) $ in the table below:

\begin{equation*}
\begin{tabular}{|l|l|l|l|l|}
\hline\hline
& $G$ & $H$ & $\limfunc{Gr}\left( u\right) $ & Ref. \\ \hline\hline
(1) & Baire &  & Suslin & \cite{Bourbaki-TG-II} \\ \hline\hline
(2) & $\beta $ & Suslin & Borel & \cite{Martineau-Studia} \\ \hline\hline
(3) & $K.\beta $ & $K$-Suslin & $F$-Borel & \cite{Martineau-Studia} \\ 
\hline\hline
(4) & 2nd category & metrizable, complete, Lindel\"{o}f & closed & \cite%
{Wilhelm} \\ \hline\hline
(5) & Suslin, Baire & Suslin & Borel &  \\ \hline\hline
(6) & Polish & Polish & closed & \cite{Hofmann} \\ \hline
\end{tabular}%
\end{equation*}

\begin{remark}
(A). \ The main statements are (1)-(4). \ Indeed,\newline
(1)$\Rightarrow $(5): Let $G$ and $H$ be Suslin. \ Then $G\times H$ is
Suslin. \ If $\limfunc{Gr}\left( u\right) $ is Borel in $G\times H,$ then $%
\limfunc{Gr}\left( u\right) $ is Suslin. \ If in addition $G$ is Baire, by
(1) $u$ is continuous.\newline
(2)$\Rightarrow $(5)$\Rightarrow $(6). If $G$ is Suslin Baire, it is $\beta
, $ thus (2)$\Rightarrow $(5). \ If $G$ and $H$ are Polish, they are Suslin
Baire, and if $\limfunc{Gr}\left( u\right) $ is closed, it is Borel, thus (5)%
$\Rightarrow $(6).\newline
(3)$\Rightarrow $(6): If $G$ and $F$\ are Polish, they are respectively $%
K.\beta $ and $K$-Suslin, and if $\limfunc{Gr}\left( u\right) $ is closed,
it is $F$-Borel. \ Therefore by (3) $u$ is continuous.\newline
(4)$\Rightarrow $(6): Let $G$ be Polish, hence Baire, hence 2nd category. \
Let $G$ be Polish, hence metrizable, complete and Lindel\"{o}f. \ By (5), if 
$\limfunc{Gr}\left( u\right) $ is closed, $u$ is continuous.\newline
(B). \ Let $E,F$ be two metrizable, complete \emph{and separable} TVS's over
a non-discrete valued division ring and $u:E\rightarrow F$ be linear with
closed graph. \ By (6), $u$ is continuous. \ In Banach's Closed Graph
Theorem, $E$ and $F$ \emph{are not assumed to be separable} thus cannot be
deduced of any of the above statements which, apart from (3), all imply a
separability condition. In addition, metrizable TVS is not $K$-Suslin in
general.
\end{remark}

\subsection{Open Mapping Theorems}

Let $u:H\rightarrow G$ be a surjective map where $G,H$ are Hausdorff
topological groups. \ Then $u$ is open under any of the following conditions
on $H$ and $G:$

\begin{equation*}
\begin{tabular}{|l|l|l|l|}
\hline\hline
& $G$ & $H$ & Ref. \\ \hline\hline
(2') & $\beta $ & Suslin & \cite{Martineau-Studia} \\ \hline\hline
(3') & $K.\beta $ & $K$-Suslin & \cite{Martineau-Studia} \\ \hline\hline
(4') & 2nd category & metrizable, complete, Lindel\"{o}f &  \\ \hline\hline
(6') & inexhaustible & Polish & \cite{Hofmann} \\ \hline
\end{tabular}%
\end{equation*}

\begin{remark}
(2'), (3'), (4') are consequences of (2), (3), (4) above and of Theorem \ref%
{Main thm}.\newline
(4') is not stated in \cite{Wilhelm}.\newline
(4')$\Rightarrow $(6'): If $G$ is inexhaustible, it is of 2nd category, and
if $H$ is Polish, it is metrizable, complete and Lindel\"{o}f.
\end{remark}

\section{Theorems in locally convex spaces\label{sect-thm-LCS}}

\subsection{Closed Graph Theorems\label{subsect-closed-graph-LCS}}

The Closed Graph Theorem was generalized by Dieudonn\'{e}-Schwartz \cite%
{Dieudonne-Schwartz}, K\"{o}the, Grothendieck (\cite{Grothendieck-PTT},
Introduction, IV, Th\'{e}or\`{e}me B) and then De Wilde (\cite{KotheII}, p.
57):

\begin{theorem}
\label{th-Bourbaki}Let $E$ be ultrabornological, $F$ be webbed, and the
graph of $u:E\rightarrow F$ be sequentially closed. \ Then $u$ is continuous.
\end{theorem}

\begin{remark}
The above result is completely symmetric when both $E$ and $F$ are countable
inductive limits of Fr\'{e}chet spaces (the case considered by K\"{o}the)
and \emph{a fortiori} when they are strict countable inductive limits of Fr%
\'{e}chet spaces (the case considered by Dieudonn\'{e}-Schwartz).
\end{remark}

Pt\'{a}k's Closed Graph Theorem (\cite{Ptak}, Theorem 4.8), proved in 1953,
and its generalization obtained in 1956\ by A. and W. Robertson (\cite%
{KotheII}, p. 41) can be stated as follows:

\begin{theorem}
\label{th-Ptak-1}Let $E$ be barrelled (resp. a locally convex hull of Baire
spaces) and $F$ be an infra-Pt\'{a}k space (resp. a locally convex hull of a
sequence of Pt\'{a}k spaces). \ A linear map $u:E\rightarrow F$ is
continuous if its graph is closed.
\end{theorem}

Schwartz's Borel graph theorem can be stated as follows (\cite{De Wilde}, p.
136, Prop. VII.4.1):

\begin{theorem}
Let $E$ be \emph{ultrabornological}, $F$ be \emph{Suslin}, and the graph of $%
u:E\rightarrow F$ is be a sq-Borel subset of $E\times F$. Then $u$ is
continuous.
\end{theorem}

In Schwartz's Borel graph theorem, $F$ is necessarily separable. \
Martineau's and Valdivia'\ generalizations do not require this restriction.
\ Martineau's result (\cite{Martineau-Studia}, Thm. 6) involves $K$%
-ultrabornological spaces which are locally convex hulls of Baire spaces of
a particular type, so this theorem is a consequence of Valdivia's result (3)
in the table below.

\begin{theorem}
Let $E,F$ be two LCS's and $u:E\rightarrow F$. \ Then $u$ is continuous if
one of the following conditions hold:\newline
$\bullet $ $E$ is the locally convex hull of Baire spaces (this happens if $%
E $ is $K$-ultrabornological), $F$ is $K$-Suslin, $u$ has closed graph (\cite%
{Valdivia-book}, p. 58, (11)).\newline
$\bullet $ $E$ is the locally convex hull of metrizable Baire spaces, $F$ is
quasi-Suslin, $u$ has closed graph (\cite{Valdivia-book}, p. 63, (17)).%
\newline
$\bullet $ $E$ is the locally convex hull of metrizable Baire (resp.
metrizable convex-Baire) spaces, $E$ is Suslin (resp. semi-Suslin), $u$ has
sequentially closed graph (\cite{Valdivia-book}, p. 71, (15) (resp. p. 84,
(13))).
\end{theorem}

Let us summarize these results in a table where the larger the row number,
the stronger the condition on $E$. \ This condition could hardly be weaker
than barrelledness since, as proved by Mahowald in 1961 (see \cite{KotheII},
p. 38), if every linear mapping of the LCS $E$ into an arbitrary Banach
space is continuous, then $E$ is barrelled.%
\begin{equation*}
\begin{tabular}{|l|l|l|l|}
\hline\hline
& $E$ & $F$ & $\limfunc{Gr}\left( u\right) $ \\ \hline\hline
(1) & barrelled & infra-Pt\'{a}k & closed \\ \hline\hline
(2) & l.c.h. Baire & $K$-Suslin & closed \\ \hline\hline
(3) & l.c.h. Baire & l.c.h. Pt\'{a}k & closed \\ \hline\hline
(4) & l.c.h. metrizable Baire & quasi-Suslin & closed \\ \hline\hline
(5) & l.c.h. metrizable Baire & Suslin & sq-closed \\ \hline\hline
(6) & l.c.h. metrizable convex-Baire & semi-Suslin & sq-closed \\ 
\hline\hline
(7) & ultrabornological & webbed & sq-closed \\ \hline\hline
(8) & ultrabornological & Suslin & sq-Borel \\ \hline
\end{tabular}%
\end{equation*}

The sq-closedness of $\limfunc{Gr}\left( u\right) $ is hardly a restriction
in practice. \ If only maps with sq-closed graph are considered, (4)$%
\Rightarrow $(5), (7)$\Rightarrow $(8).

\subsection{Open Mapping Theorems}

Let $u:F\rightarrow E$ be a surjective morphism. \ In the 3rd column of the
above table $F$ and $F/H$ $\left( H\subset F\right) $ have the same property
except for row (1) since a quotient of an infra-Pt\'{a}k space need not be
infra-Pt\'{a}k. \ Therefore, by Theorem \ref{Main thm}, if $E$ and $F$ are
as in this table, except in row (1), $u$ is open. \ Row (1) must be changed
to:

\begin{equation*}
\begin{tabular}{|l|l|l|}
\hline\hline
& $E$ & $F$ \\ \hline\hline
(1') & barrelled & Pt\'{a}k \\ \hline
\end{tabular}%
\end{equation*}%
which is Pt\'{a}k's result (\cite{Ptak}, p. 59, Thm. (4.10)).

\end{document}